\documentclass[12pt]{article}
\usepackage{amsmath,amsthm,amsfonts,amssymb}
\usepackage{fullpage,enumerate}
\usepackage{fancybox}

\newtheorem{theorem}{Theorem}[section]
\newtheorem{defn}{Definition}[section]

\newtheorem{lemma}{Lemma}[section]

\newtheorem{example}{Example}[section]

\title{Lower and upper bounds for $H$-eigenvalues of  even order real symmetric tensors}
\author{Hongwei Jin\thanks{College of Mathematics and Econometrics,
    Hunan University, 410082, Changsha, P.R. China.
    Email:  hw-jin@hotmail.com }
\and M. Rajesh Kannan\thanks{Department of Mathematics, Technion - Israel Institute of Technology, Haifa 32000, Israel. Email: rajeshkannan1.m@gmail.com }
\and Minru Bai \thanks{College of Mathematics and Econometrics,
    Hunan University, 410082, Changsha, P.R. China. E-mail: minru-bai@163.com }
}

\begin{document}
\maketitle
\begin{abstract}
In this article, we define new classes of tensors called double $\overline{B}$-tensors, quasi-double $\overline{B}$-tensors and establish some of their properties. Using these properties, we construct new regions viz., double $\overline{B}$-intervals and quasi-double $\overline{B}$-intervals,  which contain all the $H$-eigenvalues of real even order symmetric tensors. We prove that the double $\overline{B}$-intervals is contained in the quasi-double $\overline{B}$-intervals and  quasi-double ${\overline{B}}$-intervals provide supplement  information on the Brauer-type eigenvalues inclusion set of tensors. These are analogous to the double $\overline{B}$-intervals of matrices established by J. M. Pe\~na~[On an alternative to Gerschgorin circles and ovals of Cassini, Numer. Math.
95 (2003), no. 2, 337-345.] \\

\noindent{\bf AMS Subject Classification(2010):} 15A18, 15A69, 65F15.\\

\noindent{\bf Keywords:} double $\overline{B}$-tensors, quasi-double $\overline{B}$-tensors, double ${\overline{B}}$-intervals, quasi-double ${\overline{B}}$-intervals,  Brauer-type eigenvalues inclusion theorem, $H$-eigenvalues, $Z$-tensors.

\end{abstract}

\section{Introduction}

A tensor can be regarded as a higher-order generalization of a matrix, which takes the form
\begin{equation}
  \mathcal{A}=(a_{i_1\dots i_m}), \  \  a_{i_1\dots i_m}\in\mathbb{R}, \  \  i_j\in[n] :=\{1,\dots ,n\}, \ \ j\in[m].
\end{equation}
Such a multidimensional array is called an $m$-order $n$-dimensional real tensor and the set of all  $m$-order $n$-dimensional real tensors is denoted by $\mathcal{T}(\mathbb{R}^n,m)$.  A tensor $\mathcal{A}$ is called \emph{symmetric} if its entries $a_{i_1\dots i_m}$ are invariant under any permutation of their indices $\{i_1, \dots, i_m\}$.

Let $\mathbb{R}^{n} (\mathbb{C}^n)$ denote the $n$-dimensional real (complex) vector space. Vectors  are denoted by lower case letters ($x, y, \dots$), matrices by upper case letters ($A, B, \dots$) and tensors by calligraphic upper case letters ($\mathcal{A}, \mathcal{B}, \dots$). The $i^{th}$ entry of a vector $x$ is denoted by $x_i$, the $(i,j)^{th}$ entry of a matrix $A$ is denoted by $a_{ij}$  and the $(i_1,\dots, i_m)^{th}$ entry of a tensor $\mathcal{A}$ is denoted by $a_{i_1\dots i_m}$.

Given a vector $x=(x_1,\dots ,x_n)^T\in\mathbb{C}^n$, we define $\mathcal{A}x^{m-1}$ to be a vector in $\mathbb{C}^n$ whose $i$th coordinate is
\begin{equation}
  (\mathcal{A}x^{m-1})_i=\sum^n_{i_2,\dots ,i_m=1}a_{ii_2\dots i_m}x_{i_2}\cdot\cdot\cdot x_{i_m}.
\end{equation}

For $ x \in \mathbb{C}^n $ and a natural number $k$, the vector $x^{[k]}$  is the Hadamard power of $x$, i.e. $x^{[k]}_i = x^{k}_i$ for all $i$.

\begin{defn}\cite{Qi}
{\rm
Let $\mathcal{A}\in\mathcal{T}(\mathbb{R}^n,m)$. If there is a nonzero vector $x\in\mathbb{C}^n$ and a number $\lambda\in\mathbb{C}$ such that $$\mathcal{A}x^{m-1}=\lambda x^{[m-1]}$$
 , then $\lambda$ is called an \emph{eigenvalue} of the tensor $\mathcal{A}$ and $x$ an \emph{eigenvector} of $\mathcal{A}$ associated with $\lambda$. Furthermore, we say $\lambda$ is an \emph{$H$-eigenvalue} with the corresponding \emph{$H$-eigenvector} of $\mathcal{A}$ if they are real.
}
\end{defn}

 A symmetric tensor $\mathcal{A}\in\mathcal{T}(\mathbb{R}^n,m)$ is said to be  \emph{positive semidefinite} if for any vector $x \in \mathbb{R}^n$, $\mathcal{A}x^m  :=   \sum\limits_{i_1, \dots , i_m =1}^{n} a_{i_1\dots i_m}x_{i_1}\cdots x_{i_m} \geq 0$;  $\mathcal{A}$ is  \emph{positive definite} if for any nonzero vector $x \in \mathbb{R}^n$, $\mathcal{A}x^{m} >0$. From the definition it is clear that, if $m$ is odd, then there is no nontrivial positive semidefinite tensor.

Positive definite homogenous polynomials and positive semidefinite polynomials (nonnegative polynomials) are important in the field of dynamical systems, optimization, etc. With each homogenous polynomial we can associate a symmetric tensor. Checking the positive (semi)definiteness of a homogenous polynomial is equivalent to  checking the positive (semi)definiteness of the symmetric tensor  associated with it.  For details about the applications we refer to \cite{nie}, \cite{Qi} and the references therein.

Qi characterized the positive definite and positive semidefinite tensors in terms of their $H$-eigenvalues.

\begin{theorem}\cite[Theorem~5]{Qi}
Let $\mathcal{A}$ be an $m$-order $n$-dimensional symmetric tensor such that $m$ is an even integer. Then $\mathcal{A}$ is positive definite (positive semidefinite) if and only if all its $H$-eigenvalues are positive (nonnegative).
\end{theorem}

Thus, from the above theorem, the location of the $H$-eigenvalues of an even order symmetric tensor is useful in checking the positive definiteness (positive semidefiniteness) of tensors. The main purpose of this article is to give  upper and lower bounds for the $H$-eigenvalues of even order symmetric tensors. Next we recall a couple of known results in this direction.

For each $i\in[n]$, denote
\begin{equation}\label{m1}
r_i(\mathcal{A})=\sum\big\{|a_{ii_2\dots i_m}| : i_j\in[n], j=2,\dots ,m, (i_2,\dots ,i_m)\neq(i,\dots ,i)\big\}.
\end{equation}

Qi et al. \cite{Qi,QS} asserted that the eigenvalues of a tensor $\mathcal{A}$ have a similar statement as the Gerschgorin circle of the eigenvalues of matrices. One can rewrite the result for $H$-eigenvalues of a real tensor as follows.

\begin{theorem}\label{sg}\cite[Theorem \ 2]{QS}
Let $\mathcal{A}=(a_{i_1\dots i_m})\in \mathcal{T}(\mathbb{R}^n,m)$ and $\lambda$ be an $H$-eigenvalue of $\mathcal{A}$. Then,
\begin{equation}\label{dkkd}
 \lambda\in\Gamma(\mathcal{A})=\bigcup_{i=1}^n\Gamma_i(\mathcal{A}),
\end{equation}
where $\Gamma_i(\mathcal{A})=\Big[a_{i\dots i}-r_i(\mathcal{A}), \ \ a_{i\dots i}+r_i(\mathcal{A})\Big]$.
\end{theorem}

We call the interval $\Gamma(\mathcal{A})$ in $(\ref{dkkd})$ the Gerschgorin eigenvalues inclusion set of tensors. Li et al.  established a Brauer-type eigenvalues inclusion set for an arbitrary complex tensor \cite[Theorem 2.1]{LLK} and  showed that the Brauer-type eigenvalues inclusion set is contained in the Gerschgorin eigenvalues inclusion set  \cite[Theorem 2.3]{LLK}. We can restate the Brauer-type eigenvalues inclusion theorem for $H$-eigenvalues of a real tensor as follows:

\begin{theorem}\label{s0g}\cite[Theorem \ 2.1]{LLK}
Let $\mathcal{A}=(a_{i_1\dots i_m})\in \mathcal{T}(\mathbb{R}^n,m)$ and $\lambda$ be an $H$-eigenvalue of $\mathcal{A}$. Then,
\begin{equation}\label{dnnd}
 \lambda\in\Omega(\mathcal{A})=\bigcup_{i,j=1\atop i\neq j}^n\Omega_{ij}(\mathcal{A}),
\end{equation}
where $$\Omega_{ij}(\mathcal{A})=\{z\in\mathbb{C} : |z-a_{i\dots i}|(|z-a_{j\dots j}|-r_j^i(\mathcal{A}))\leq r_i(\mathcal{A})|a_{ji\dots i}|\},$$
 and
 $$r_j^i(\mathcal{A})=r_j(\mathcal{A})-|a_{ji\dots i}|=
\sum\limits_{j_2,\dots ,j_m=1\atop{(j_2,\dots ,j_m)\neq(i,\dots ,i)\atop(j_2,\dots ,j_m)\neq(j,\dots ,j)}}^n|a_{jj_2\dots j_m}|.$$
\end{theorem}

In this article, in Section \ref{notations}, we collect some definitions, known results and correct minor mistakes in couple of results proved in \cite{LL}.  In Section \ref{dt-qdt},  we define two new classes tensors viz.,  double $\overline{B}$-tensors, quasi-double $\overline{B}$-tensors and establish their properties. Using this properties, in Section \ref{eigen-region}, we construct two new regions called double ${\overline{B}}$-intervals and quasi-double ${\overline{B}}$-intervals, containing all the $H$-eigenvalues of even order symmetric tensors. These regions have a nature similar to the Brauer-type eigenvalues inclusion set $\Omega(\mathcal{A})$ stated in $(\ref{dnnd})$.  We prove that quasi-double ${\overline{B}}$-intervals is smaller than  double ${\overline{B}}$-intervals and there is no inclusion relation between  quasi-double ${\overline{B}}$-intervals and the Brauer-type eigenvalues inclusion set. Hence, it is nature to construct an intervals $\Upsilon(\mathcal{A})$, narrower than both quasi-double ${\overline{B}}$-intervals and  Brauer-type eigenvalues inclusion set $\Omega(\mathcal{A})$, by intersecting the two. These comparison results are done in Section \ref{compare}. In Section \ref{conclude}, we draw some concluding remarks.

\section{Notation, Definitions and Preliminary results}\label{notations}

The purpose of this section is twofold -  to collect  known definitions and results; and  to correct  minor mistakes in couple of results in \cite{LL}.

The $m$-order $n$-dimensional \emph{identity tensor}, denoted by $\mathcal{I}$, is the tensor with entries
$$a_{i_1\dots i_m}=\left\{
                  \begin{array}{ll}
                    1, & \text{if} \ \ i_1=\dots=i_m, \\
                    0, & \text{otherwise}.
                  \end{array}
                \right.
$$

Chen et al. \cite{CQ} defined the following $k$th row tensor dealing with some problems of circulant tensors.
\begin{defn}\cite{CQ}
{\rm
Let $\mathcal{A}=(a_{i_1\dots i_m})\in  \mathcal{T}(\mathbb{R}^n,m)$. Then, $\mathcal{A}_k=(a^{(k)}_{i_1\dots i_{m-1}})\in\mathcal{T}(\mathbb{R}^n,m-1)$ is called  the \emph{$k$th row tensor} if $a^{(k)}_{i_1\dots i_{m-1}}=a_{ki_1\dots i_{m-1}}$, where $k, i_1,\dots ,i_{m-1} \in[n]$.}
\end{defn}

Let $sign(x)$ be sign function, that is, \begin{equation*}
 sign(x)=\left\{
        \begin{array}{lll}
          1, &  \  x>0, \\
         0, &  \  x=0,\\
        -1, &  \  x<0.\\
\end{array}
      \right.
\end{equation*}

For a tensor $\mathcal{A}=(a_{i_1\dots i_m})\in  \mathcal{T}(\mathbb{R}^n,m)$ and  $i\in[n]$, define
\begin{equation}\label{88}
\beta_i(\mathcal{A})= \max\{0, \  a_{ii_2\dots i_m} \   : \   (i_2,\dots ,i_m)\neq(i,\dots ,i)\},
\end{equation}
\begin{equation}\label{99}
\gamma_i(\mathcal{A})= \min\{0,  \  a_{ii_2\dots i_m} \   :  \  (i_2,\dots ,i_m)\neq(i,\dots ,i)\},
\end{equation}
\begin{equation}\label{77}
\Delta_i(\mathcal{A}) = \sum\{(\beta_i (\mathcal{A}) - a_{ii_2\dots i_m}) : \ \  (i_2,\dots ,i_m)\neq(i,\dots ,i) \},
\end{equation}
\begin{equation}\label{66}
\Delta_j^i(\mathcal{A}) = \sum\{(\beta_{j}(\mathcal{A})-a_{jj_2\dots j_m}): (j_2,\dots ,j_m)\neq(i,\dots ,i), (j_2,\dots ,j_m)\neq(j,\dots ,j)\},
\end{equation}
\begin{equation}\label{55}
\Theta_i(\mathcal{A}) = \sum\{( a_{ii_2\dots i_m}  - \gamma_i (\mathcal{A})): \ \  (i_2,\dots ,i_m)\neq(i,\dots ,i) \}
\end{equation}
and
\begin{equation}\label{44}
\Theta_j^i(\mathcal{A}) = \sum\{ (a_{jj_2\dots j_m}  - \gamma_j (\mathcal{A})): \ \  (j_2,\dots ,j_m)\neq(i,\dots ,i), (j_2,\dots ,j_m)\neq(j,\dots ,j) \}.
\end{equation}

Next we recall the definitions of double $B$-tensor and quasi-double ${B}$-tensor defined by Li  et al. in \cite{LL}.

\begin{defn}\label{doubleb}
{\rm
Let $\mathcal{A}=(a_{i_1\dots i_m})\in  \mathcal{T}(\mathbb{R}^n,m)$ with $a_{i\dots i} > \beta_{i}(\mathcal{A})$ for all $i$. Then $\mathcal{A}$ is  said to be a \emph{double $B$-tensor} if :

\begin{itemize}
\item[(a)] for any $i \in \{1,\dots, n\}, a_{i\dots i} - \beta_i(\mathcal{A}) \geq \Delta_i(\mathcal{A})$,
\item[(b)] for all $i, j \in \{1, \dots , n\}, i \neq j, (a_{i\dots i}-\beta_{i}(\mathcal{A}))(a_{j\dots j}-\beta_{j}(\mathcal{A})) > \Delta_i(\mathcal{A})\Delta_j(\mathcal{A}).$
\end{itemize}
}
\end{defn}

Note that Chen et al. $\cite{CYY}$ also gave a definition of double $B$-tensor, which is different from the  above one. A double $B$-tensor defined in \cite{CYY} need not satisfy the condition $(a)$ of definition \ref{doubleb}. If we drop condition $(a)$ from definition \ref{doubleb}, then even order symmetric double $B$-tensors need not be positive definite\cite{LL}. Positive definiteness of double $B$-tensors plays a crucial role in locating the $H$-eigenvalues of even order symmetric tensors. So, in this paper, we follow the definition of double $B$-tensor as in \cite{LL}.

\begin{defn}\cite{LL}
{\rm
Let $\mathcal{A}=(a_{i_1\dots i_m})\in  \mathcal{T}(\mathbb{R}^n,m)$ with $a_{i\dots i} > \beta_{i}(\mathcal{A})$ for all $i$. Then $\mathcal{A}$ is  said to be a \emph{quasi-double $B$-tensor} if :
$$(a_{i\dots i}-\beta_{i}(\mathcal{A}))(a_{j\dots j}-\beta_{j}(\mathcal{A})-\Delta_j^i(\mathcal{A})) > (\beta_i(\mathcal{A}) - a_{ji\dots i} )\Delta_j(\mathcal{A}).$$
}
\end{defn}

Next we recall some definitions.

\begin{defn}$\cite{LL}$
{\rm
Let $\mathcal{A}$ be an $m$-order $n$-dimensional tensor. Then $\mathcal{A}$ is called a \emph{doubly strictly diagonally dominant tensor (DSDD)} if:
\begin{itemize}
\item[(a)]  $|a_{i\dots i}||a_{j\dots j}| > r_i(\mathcal{A})r_j(\mathcal{A}), ~\mbox{for all }~ i,j , i \neq j,$
\item[(b)] when $m>2$, $|a_{i \dots i}| \geq r_i(\mathcal{A})$ for $i \in \{1,\dots, n\}$ .
\end{itemize}
}
\end{defn}

\begin{defn}$\cite{LL}$
{\rm Let $\mathcal{A}$ be an $m$-order $n$-dimensional tensor. Then $\mathcal{A}$ is called a \emph{quasi-doubly strictly diagonally dominant tensor (Q-DSDD)} if: $$|a_{i\dots i}|(|a_{j\dots j}|-r_j^i(\mathcal{A})) > r_i(\mathcal{A})|a_{ji\dots i}|, ~\mbox{for all }~ i,j , i \neq j. $$}
\end{defn}
A tensor $\mathcal{A}$ is called a \emph{$Z$-tensor} if there exists a tensor $\mathcal{D}$ with nonnegative entries and a real number $s$ such that $\mathcal{A} = s\mathcal{I} - \mathcal{D}$.

Li et al. \cite{LL} proved that the following important result using the fact that an even order symmetric double ${B}$-tensor (quasi-double ${B}$-tensor) can be decomposed into the sum of a doubly (quasi-doubly) strictly diagonally dominant symmetric ${Z}$-tensor and several positive multiples of partially all one tensors.

\begin{theorem}$\cite{LL}$\label{4l1} The following statements are true:
\begin{itemize}
\item[(a)] All the $H$-eigenvalues of an even order symmetric double ${B}$-tensor are positive,
\item[(b)] All the $H$-eigenvalues of an even order symmetric quasi-double ${B}$-tensor are positive.
\end{itemize}
\end{theorem}

The following results are proved in \cite{LL}.  We observed that these results contain minor mistakes. We give a counter example to these results and prove the correct versions (Theorem \ref{correct-li-li}, Theorem \ref{modified-positive-li-li}).
\begin{theorem}\cite[Proposition~5]{LL} \label{result-li-li}
Let $\mathcal{A}$ be an  $m$-order $n$-dimensional $Z$-tensor. Then:
\begin{itemize}
\item[(a)] $\mathcal{A}$ is a double $B$-tensor if and only if $\mathcal{A}$ is a DSDD tensor.
\item[(b)] $\mathcal{A}$ is a quasi-double $B$-tensor if and only if $\mathcal{A}$ is a Q-DSDD tensor.
\end{itemize}
\end{theorem}

\begin{theorem}\cite[Theorem 4]{LL}\label{positive-li-li}
Let $\mathcal{A}$ be an  $m$-order $n$-dimensional real symmetric tensor such that $m$ is even.
If $\mathcal{A}$  is either a DSDD tensor or a Q-DSDD tensor, then $\mathcal{A}$ is positive definite.
\end{theorem}

 The following example shows that Theorem \ref{result-li-li} and \ref{positive-li-li} are not true. Consider the $Z$-matrix $A = \left(
\begin{array}{ccc}
-1  & -\frac{1}{2} \\
-\frac{1}{2}  & -1\end{array} \right)$. Then $A$ is both DSDD and  Q-DSDD-matrix. But, it is neither a double $B$-matrix nor a quasi-double $B$-matrix. Also it is easy to see that $A$ is not positive definite.

We will give a correct version of Theorem \ref{result-li-li} in Section \ref{dt-qdt}. Now we give a correct version of Theorem \ref{positive-li-li} as the follows.

\begin{theorem}\label{modified-positive-li-li}
Let $\mathcal{A}$ be an  $m$-order $n$-dimensional real symmetric tensor such that $m$ is even and $a_{i\dots i} > 0$ for all $i \in [n]$.
\begin{itemize}
\item[(1)] If $\mathcal{A}$  is a DSDD tensor, then $\mathcal{A}$ is positive definite;
\par
\item[(2)] If $\mathcal{A}$ is  a Q-DSDD tensor, then $\mathcal{A}$ is positive definite;
\end{itemize}
\end{theorem}
Proof. The conclusion (1) follows directly from \cite[Theorem 11]{li-wang-zhao-zhu-li}. If $\mathcal{A}$ is  a Q-DSDD tensor, then it is easy to see that there exist an index $i\in [n]$ such that $|a_{i\dots i}| > r_i(\mathcal{A})$. Now the conclusion (2) follows directly from \cite[Theorem 13]{li-wang-zhao-zhu-li}.

\section{Double $\overline{B}$-tensors and quasi-double $\overline{B}$-tensors}\label{dt-qdt}

In this section, first we introduce the notion of  double $\overline{B}$-tensor and quasi-double $\overline{B}$-tensor.
\begin{defn}\label{dd5}
{\rm
Let $\mathcal{A}=(a_{i_1\dots i_m})\in  \mathcal{T}(\mathbb{R}^n,m)$ be a tensor with the $k$th row tensor $\mathcal{A}_k$, $k\in[n]$ and $\mathcal{\overline{A}}=(\overline{a}_{i_1\dots i_m})\in  \mathcal{T}(\mathbb{R}^n,m)$ be a tensor with the $k$th row tensor sign$(a_{k\dots k})$$\mathcal{A}_k$, $k\in[n]$. Then,
\begin{itemize}
\item[(a)]  $\mathcal{A}$ is called a double $\overline{B}$-tensor if $\mathcal{\overline{A}}$ is a double $B$-tensor,
\item[(b)] $\mathcal{A}$ is called a quasi-double ${\overline{B}}$-tensor if $\mathcal{\overline{A}}$ is a quasi-double ${B}$-tensor.
\end{itemize}
}
\end{defn}

Note that,  double $\overline{B}$-tensors discussed in this paper is different from double $\overline{B}$-tensors discussed in \cite{CYY} because of the different definitions of double $B$-tensor.

Let $\mathcal{A}$ be a double ${{B}}$-tensor. Since $a_{i\dots i}>\beta_i(\mathcal{A})\geq0$ for each $i\in[n]$, it is easy to see that $\mathcal{\overline{A}}$ is a double ${B}$-tensor, which means $\mathcal{A}$ is a double ${\overline{B}}$-tensor. Hence,
\begin{equation*}
\{\text{double} \ {{B}}\text{-tensors}\}\subset\{\text{double} \ {\overline{B}}\text{-tensors}\}.
\end{equation*}
Similarly, we have
\begin{equation*}
\{\text{quasi-double} \ {{B}}\text{-tensors}\}\subset\{\text{quasi-double} \ {\overline{B}}\text{-tensors}\}.
\end{equation*}

Define $$\delta_{i_1\dots i_m}=\left\{
                  \begin{array}{ll}
                    1, & \text{if} \ \ i_1=\dots=i_m, \\
                    0, & \text{otherwise}.
                  \end{array}
                \right.
$$

Next we prove a correct version of  Theorem \ref{result-li-li}.

\begin{theorem}\label{correct-li-li}
Let $\mathcal{A}$ be an  $m$-order $n$-dimensional tensor such that $\overline{\mathcal{A}}$ is a $Z$-tensor. Then
\begin{itemize}
\item[(a)] $\mathcal{A}$ is a double $\overline{B}$-tensor if and only if $\mathcal{A}$ is a DSDD tensor.
\item[(b)] $\mathcal{A}$ is a quasi-double $\overline{B}$-tensor if and only if $\mathcal{A}$ is a Q-DSDD tensor.
\end{itemize}
\end{theorem}
\begin{proof}\textbf{(a)}
Let $\mathcal{B} = \overline{\mathcal{A}}$.
Since $\mathcal{B}$ is a  $Z$-tensor, we have $\beta_i(\mathcal{B}) = 0 $  and  $r_i(\mathcal{B}) = \Delta_i(\mathcal{B})= \sum\limits_{\delta_{ii_2\dots i_m}= 0} -b_{ii_2\dots i_m} $ for all $i$. Hence, $$|b_{i\dots i}||b_{j \dots j}| > r_i(\mathcal{B}) r_j(\mathcal{B})$$ if and only if $$(b_{i\dots i } - \beta_i(\mathcal{B}))(b_{j\dots j } - \beta_j(\mathcal{B})) > \Delta_i(\mathcal{B})\Delta_j(\mathcal{B}).$$ Also $b_{i\dots i} - \beta_i(\mathcal{B}) \geq \Delta_i(\mathcal{B})$ if and only if $|a_{i\dots i}| \geq r_i(\mathcal{A}).$  Thus $\mathcal{A}$ is a double $\overline{B}$-tensor if and only if $\mathcal{A}$ is a DSDD tensor.

\textbf{(b)}Now, $\Delta_j^i(\mathcal{B}) = \sum\limits_{\substack{\delta_{ij_2\dots j_m} = 0 \\ \delta_{jj_2\dots j_m} = 0}}(\beta_{j}(\mathcal{B})-b_{jj_2\dots j_m}) = r_j^i(\mathcal{B}).$ Hence $$b_{i\dots i} (b_{j\dots j} - r_j^i(\mathcal{B})) > r_i(\mathcal{B})|b_{ji\dots i}|$$ if and only if $$(b_{i\dots i} - \beta_{i}(\mathcal{B}))(b_{j\dots j} - \beta_j(\mathcal{B})- \Delta_j^i(\mathcal{B}))>(\beta_j(\mathcal{B}) - b_{ji\dots i})\Delta_i(\mathcal{B}).$$ Thus $\mathcal{A}$ is a quasi double $\overline{B}$-tensor if and only if $\mathcal{A}$ is a Q-DSDD tensor.
\end{proof}

Let us recall the general product of two $n$-dimensional tensors defined by Shao in $\cite{Shao}$.
\begin{defn}$\cite{Shao}$
Let $\mathcal{A}=(a_{i_1\dots i_m})\in\mathcal{T}(\mathbb{C}^n,m)$ and $\mathcal{B}=(b_{i_1\dots i_k})\in\mathcal{T}(\mathbb{C}^n,k)$.  Define the product $\mathcal{A}\mathcal{B}$ to be the following tensor $\mathcal{C}$ of
order $(m-1)(k-1)+1$ and dimension $n$:
\begin{equation*}
c_{i\alpha_1\dots \alpha_{m-1}}=\sum^n_{i_2,\dots ,i_m=1}a_{ii_2\dots i_m}b_{i_2\alpha_1}\cdot\cdot\cdot b_{i_m\alpha_{m-1}}, \  \ (i\in[n], \alpha_1,\dots ,\alpha_{m-1}\in[n]^{k-1}).
\end{equation*}
\end{defn}

By the definition of double $\overline{B}$-tensor and quasi-double ${\overline{B}}$-tensor, it is easy to prove the following.

\begin{theorem}\label{ld6.3}
Let $\mathcal{A}=(a_{i_1\dots i_m})\in \mathcal{T}(\mathbb{R}^n,m)$. Then
\begin{itemize}
\item[(a)] $\mathcal{A}$ is a double $\overline{B}$-tensor if and only if there exists an $n\times n$ diagonal matrix $D$ whose diagonal elements belong to the set $\{1,-1\}$ and
a double ${B}$-tensor $\mathcal{B}\in \mathcal{T}(\mathbb{R}^n,m)$ such that $\mathcal{A}=D\mathcal{B}$,
\item[(b)] $\mathcal{A}$ is a quasi-double ${\overline{B}}$-tensor if and only if there exists an $n\times n$ diagonal matrix $D$ whose diagonal elements belong to the set $\{1,-1\}$ and
a quasi-double ${B}$-tensor $\mathcal{B}\in \mathcal{T}(\mathbb{R}^n,m)$ such that $\mathcal{A}=D\mathcal{B}$.
\end{itemize}

\end{theorem}

\begin{defn}
{\rm Let $\mathcal{A}=(a_{i_1\dots i_m})$ be an $m$-order $n$-dimensional tensor and $\alpha \subseteq\{1, \dots , n\}$ with $|\alpha| = r$.
A \emph{principal subtensor} $\mathcal{A}[\alpha]$ of the tensor $\mathcal{A}$ with index set $\alpha$ is an $m$-order $r$-dimensional subtensor of  $\mathcal{A}$ consisting of $r^m$ elements defined as
follows:

\begin{center}
$\mathcal{A}[\alpha] = (a_{i_1\dots i_m})$ , where $i_1, \dots , i_m \in \alpha$.
\end{center}
}
\end{defn}
\begin{theorem}
Let $\mathcal{A}$ be an $m$-order $n$-dimensional tensor.
\begin{itemize}
\item[(a)] If $\mathcal{A}$ is a double $\overline{B}$-tensor, then $\mathcal{A}[\alpha]$ is a double $\overline{B}$-tensor for all $\alpha \subseteq \{1,\dots, n\}$.
\item[(b)] If $\mathcal{A}$ is a quasi-double $\overline{B}$-tensor, then $\mathcal{A}[\alpha]$ is a quasi-double $\overline{B}$-tensor for all $\alpha \subseteq \{1,\dots, n\}$.
\end{itemize}
\end{theorem}

\begin{proof}
\textbf{(a)} Let $\mathcal{B} = \mathcal{\overline{A}}$, $\alpha \subseteq \{1,\dots, n\}$ and  $i \in \alpha$. Since $b_{i\dots i} > \beta_i(\mathcal{B})$, we have $b_{i\dots i } > \beta_i(\mathcal{B}[\alpha])$. Also, $b_{i\dots i}-\beta_i(\mathcal{B}) \geq  \Delta_i(\mathcal{B})$ and $\Delta_i(\mathcal{B}) \geq \Delta_i(\mathcal{B}[\alpha])$.  Thus $b_{i\dots i}-\beta_i(\mathcal{B}[\alpha]) \geq \Delta_i(\mathcal{B}[\alpha]).$ Now,
\begin{align*}
(b_{i\dots i} - \beta_i(\mathcal{B}[\alpha]))(b_{j\dots j} - \beta_j(\mathcal{B}[\alpha]))  &\geq (b_{i\dots i} - \beta_i(\mathcal{B}))(b_{j\dots j} - \beta_j(\mathcal{B}))  \\
  &>    \Delta_i(\mathcal{B}) \Delta_j(\mathcal{B})\\
  &\geq  \Delta_i(\mathcal{B}[\alpha]) \Delta_j(\mathcal{B}[\alpha]) .
  \end{align*}
  Hence $\mathcal{A}[\alpha]$ is double $\overline{B}$-tensor.

\textbf{(b)}  Let $i, j \in \alpha$ and $i \neq j$. We have $\Delta_j^i(\mathcal{B}) \geq \Delta_j^i(\mathcal{B}[\alpha])$. Now,
\begin{align*}
(b_{i\dots i} - \beta_i(\mathcal{B}[\alpha]))(b_{j\dots j} - \beta_j(\mathcal{B}[\alpha])- \Delta_j^i(\mathcal{B}[\alpha]))  &\geq (b_{i\dots i} - \beta_i(\mathcal{B}))(b_{j\dots j} - \beta_j(\mathcal{B})- \Delta_j^i(\mathcal{B}))  \\
  &>     (\beta_j(\mathcal{B})- b_{ji\dots i}) \Delta_i(\mathcal{B})\\
  &\geq  (\beta_j(\mathcal{B}[\alpha])- b_{ji\dots i}) \Delta_i(\mathcal{B}[\alpha]).
  \end{align*}
  Thus $\mathcal{A}[\alpha]$ is a quasi-double $\overline{B}$-tensor.
\end{proof}

\begin{defn}
\rm{Let $\mathcal{A}$ be an $m$-order $n$-dimensional tensor. Then $\mathcal{A}^+ =(a^+_{i_1\dots i_m}) $ is an $m$-order $n$-dimensional tensor defined as $a^+_{i_1\dots i_m} = a_{i_1\dots i_m} - \beta_{i_1}(\mathcal{A})$.}
\end{defn}
It is easy to see that $\mathcal{A}^+$ is a $Z$-tensor.
\begin{theorem}
Let $\mathcal{A}$ be an $m$-order $n$-dimensional tensor. Then
\begin{itemize}
\item[(a)] $\mathcal{A}$ is a double $\overline{B}$-tensor if and only if $\overline{\mathcal{A}}^{+}$ is a double $B$-tensor.
\item[(b)] $\mathcal{A}$ is a quasi-double $\overline{B}$-tensor if and only if $\overline{\mathcal{A}}^{+}$ is a quasi-double $B$-tensor.
\end{itemize}
\end{theorem}
\begin{proof}
 Let $\mathcal{B} = \mathcal{\overline{A}}$ and $\mathcal{C} = \overline{\mathcal{A}}^{+}$.

  \noindent\textbf{(a)} Suppose $\mathcal{A}$ is a double $\overline{B}$-tensor. Then $c_{i\dots i} = b_{i\dots i} - \beta_i(\mathcal{B})> 0.$
Since $\mathcal{C}$ is a $Z$-tensor and $\Delta_i(\mathcal{B}) = \Delta_i(\mathcal{C})$, we have $c_{i\dots i} \geq \Delta_i(\mathcal{C})$. Now, we have
$$(c_{i\dots i} - \beta_i(\mathcal{C}))(c_{j\dots j} - \beta_j(\mathcal{C}))  = (b_{i\dots i} - \beta_i(\mathcal{B}))(b_{j\dots j} - \beta_j(\mathcal{B}))  $$
 and $$    \Delta_i(\mathcal{B}) \Delta_j(\mathcal{B})
  =   \Delta_i(\mathcal{C}) \Delta_j(\mathcal{C}).$$

  Thus $$(c_{i\dots i} - \beta_i(\mathcal{C}))(c_{j\dots j} - \beta_j(\mathcal{C}))>\Delta_i(\mathcal{C}) \Delta_j(\mathcal{C})$$ if and only if $$(b_{i\dots i} - \beta_i(\mathcal{B}))(b_{j\dots j} - \beta_j(\mathcal{B}))> \Delta_i(\mathcal{B}) \Delta_j(\mathcal{B}).$$
 Hence  $\mathcal{A}$ is a double $\overline{B}$-tensor if and only if $\overline{\mathcal{A}}^{+}$ is a double $B$-tensor.

\noindent\textbf{(b)} We have $$(c_{i\dots i} -  \beta_i(\mathcal{C})) (c_{j\dots j} - \beta_j(\mathcal{C}) - \Delta_j^i(\mathcal{C})) = (b_{i\dots i} -  \beta_i(\mathcal{B})) (b_{j\dots j} - \beta_j(\mathcal{B}) - \Delta_j^i(\mathcal{B}))$$ and $$(\beta_j(\mathcal{C})-c_{ji\dots i})\Delta_j^i(\mathcal{C}) =(\beta_j(\mathcal{B})-b_{ji\dots i})\Delta_j^i(\mathcal{B}).$$

Thus $$(c_{i\dots i} -  \beta_i(\mathcal{C})) (c_{j\dots j} - \beta_j(\mathcal{C}) - \Delta_j^i(\mathcal{C})) > (\beta_j(\mathcal{C})-c_{ji\dots i})\Delta_j^i(\mathcal{C})$$ if and only if $$(b_{i\dots i} -  \beta_i(\mathcal{B})) (b_{j\dots j} - \beta_j(\mathcal{B}) - \Delta_j^i(\mathcal{B}))> (\beta_j(\mathcal{B})-b_{ji\dots i})\Delta_j^i(\mathcal{B}).$$ Hence $\mathcal{A}$ is a quasi-double $\overline{B}$-tensor if and only if $\overline{\mathcal{A}}^{+}$ is a quasi-double $B$-tensor.
\end{proof}

Define $$\alpha_i(\mathcal{A})= \left\{
	\begin{array}{ll}
		\beta_i(\mathcal{A}) & \mbox{if } a_{i\dots i} > 0, \\
		\gamma_i(\mathcal{A})  & \mbox{if } a_{i\dots i} < 0,
	\end{array}
\right. $$
where $\beta_i(\mathcal{A})$ and $\gamma_i(\mathcal{A})$ are defined in $(\ref{88})$ and $(\ref{99})$, respectively.

The following theorem gives an equivalent condition for a  tensor to be a double ${\overline{B}}$-tensor.

\begin{theorem}\label{ta5}
Let $\mathcal{A}$ be an $m$-order $n$-dimensional tensor. Then, $\mathcal{A}$ is a double $\overline{B}$-tensor if and only if
\begin{itemize}
\item[(a)] $|a_{i\dots i}| > |\alpha_i(\mathcal{A})|$ for all $i$,
\item[(b)] $|a_{i\dots i} - \alpha_i(\mathcal{A})| \geq \sum\limits_{\delta_{ii_2\dots i_m = 0}} |\alpha_i(\mathcal{A})-a_{ii_2\dots i_m}|$,
\item[(c)] $|a_{i\dots i} - \alpha_i(\mathcal{A})||a_{j\dots j} - \alpha_j(\mathcal{A})| > (\sum\limits_{\delta_{ii_2\dots i_m = 0}} |\alpha_i(\mathcal{A})-a_{ii_2\dots i_m}|)(\sum\limits_{\delta_{ji_2\dots i_m = 0}} |\alpha_j(\mathcal{A})-a_{ji_2\dots i_m}|).$
\end{itemize}
\end{theorem}
\begin{proof}
Let $\mathcal{B} = \overline{\mathcal{A}}$. Then $b_{i\dots i} > \beta_i(\mathcal{B})$ if and only if $a_{i\dots i} > \beta_i(\mathcal{A})$ if $a_{i \dots i} > 0$ and  $-a_{i\dots i} > -\gamma_i(\mathcal{A})$ if $a_{i\dots i} <0$. Thus $b_{i\dots i} > \beta_i(\mathcal{B})$  if and only if $|a_{i\dots i}| > |\alpha_i(\mathcal{A})|$ for all $i$.

If $a_{i\dots i} < 0$, then
$$-a_{i\dots i} + \alpha_i(\mathcal{A}) =   b_{i \dots i} - \beta_i(\mathcal{B}) \geq \sum_{\delta_{ii_2\dots i_m}=0} (\beta_i(\mathcal{B}) - b_{ii_2\dots i_m})\\ = \sum_{\delta_{ii_2\dots i_m}=0} |\alpha_i(\mathcal{A}) - a_{ii_2\dots i_m}|$$
and if $a_{i\dots i} >0$, then
$$a_{i\dots i} - \alpha_i(\mathcal{A}) =   b_{i \dots i} - \beta_i(\mathcal{B}) \geq \sum_{\delta_{ii_2\dots i_m}=0} (\beta_i(\mathcal{B}) - b_{ii_2\dots i_m})\\ = \sum_{\delta_{ii_2\dots i_m}=0} |\alpha_i(\mathcal{A}) - a_{ii_2\dots i_m}|.$$ Thus $|a_{i\dots i} - \alpha_i(\mathcal{A})| \geq \sum\limits_{\delta_{ii_2\dots i_m = 0}} |\alpha_i(\mathcal{A})-a_{ii_2\dots i_m}|$ if and only if $$b_{i \dots i} - \beta_i(\mathcal{B}) \geq \sum_{\delta_{ii_2\dots i_m}=0} (\beta_i(\mathcal{B}) - b_{ii_2\dots i_m}).$$


Since
$$|a_{i\dots i}-\alpha_i(\mathcal{A})||a_{j\dots j}-\alpha_j(\mathcal{A})| = (b_{i\dots i}-\beta_i(\mathcal{B}))(b_{j\dots j}-\beta_j(\mathcal{B}))$$
and
\begin{align*}
&(\sum_{\delta_{ii_2\dots i_m}=0}(\beta_i(\mathcal{B})-b_{ii_2\dots i_m}))(\sum_{\delta_{ji_2\dots i_m}=0}(\beta_j(\mathcal{B})-b_{ji_2\dots i_m}))\\
&= (\sum_{\delta_{ii_2\dots i_m}=0}|\alpha_i(\mathcal{A})-a_{ii_2\dots i_m}|)(\sum_{\delta_{ji_2\dots i_m}=0}|\alpha_j(\mathcal{A})-a_{ji_2\dots i_m}|),
\end{align*}
we have
$$|a_{i\dots i} - \alpha_i(\mathcal{A})||a_{j\dots j} - \alpha_j(\mathcal{A})| > (\sum_{\delta_{ii_2\dots i_m = 0}} |\alpha_i(\mathcal{A})-a_{ii_2\dots i_m}|)(\sum_{\delta_{ji_2\dots i_m = 0}} |\alpha_j(\mathcal{A})-a_{ji_2\dots i_m}|)$$ if and only if $$(b_{i\dots i}-\beta_i(\mathcal{B}))(b_{j\dots j}-\beta_j(\mathcal{B}))
> (\sum_{\delta_{ii_2\dots i_m}=0}(\beta_i(\mathcal{B})-b_{ii_2\dots i_m}))(\sum_{\delta_{ji_2\dots i_m}=0}(\beta_j(\mathcal{B})-b_{ji_2\dots i_m})).$$
This proves the result.
\end{proof}
The following result gives an equivalent condition for a  tensor to be a quasi-double ${\overline{B}}$-tensor.

\begin{theorem}\label{mk}
Let $\mathcal{A}=(a_{i_1\dots i_m})$ be an $m$-order $n$-dimensional tensor. Then $\mathcal{A}$ is a quasi-double ${\overline{B}}$-tensor if and only if for
each $i\in[n]$,  $|a_{i\dots i}|>|\alpha_i(\mathcal{A})|$ and for any $i\neq j\in[n]$,

\begin{equation}\label{a2aa}
\begin{aligned}
 &  |a_{i\dots i}-\alpha_i(\mathcal{A})|(|a_{j\dots j}-\alpha_j(\mathcal{A})|-\sum\limits_{\delta_{ij_2\dots j_m = 0}\atop\delta_{jj_2\dots j_m = 0}}|\alpha_j(\mathcal{A})-a_{jj_2\dots j_m}|)\\
      & >|\alpha_j(\mathcal{A})-a_{ji\dots i}|\sum\limits_{\delta_{ii_2\dots i_m = 0}}|\alpha_i(\mathcal{A})-a_{ii_2\dots i_m}|.
\end{aligned}
\end{equation}
\end{theorem}

\begin{proof}
Let $\mathcal{B} = \overline{\mathcal{A}}$. By the definition of $\alpha_i(\mathcal{A})$, one has that $b_{i\dots i}>\beta_i(\mathcal{A})$ if and
only if $a_{i\dots i}>\beta_i(\mathcal{A})$ if $a_{i\dots i}>0$ and $a_{i\dots i}<\gamma_i(\mathcal{A})$ if  $a_{i\dots i}<0$. Thus, we have  $b_{i\dots i}>\beta_i(\mathcal{A})$ if and
only if $|a_{i\dots i}|>|\alpha_i(\mathcal{A})|$.

Since,
\begin{equation}\label{a2aa}
\begin{aligned}
 &  |a_{i\dots i}-\alpha_i(\mathcal{A})|(|a_{j\dots j}-\alpha_j(\mathcal{A})|-\sum\limits_{\delta_{ij_2\dots j_m = 0}\atop\delta_{jj_2\dots j_m = 0}}|\alpha_j(\mathcal{A})-a_{jj_2\dots j_m}|)\\
      &= (b_{i\dots i}-\beta_i(\mathcal{B})(b_{j\dots j}-\beta_j(\mathcal{B})-\sum\limits_{\delta_{ij_2\dots j_m = 0}\atop\delta_{jj_2\dots j_m = 0}}(\beta_j(\mathcal{B})-b_{jj_2\dots j_m})
\end{aligned}
\end{equation}

if and only if
\begin{equation}\label{a2aa}
\begin{aligned}
 &   (\beta_j(\mathcal{B})-b_{ji\dots i})\sum\limits_{\delta_{ii_2\dots i_m = 0}}(\beta_i(\mathcal{B})-b_{ii_2\dots i_m}) \\
      &= |\alpha_j(\mathcal{A})-a_{ji\dots i}|\sum\limits_{\delta_{ii_2\dots i_m = 0}}|\alpha_i(\mathcal{A})-a_{ii_2\dots i_m}|
\end{aligned}
\end{equation}

Thus, we have

\begin{equation}\label{a2aa}
\begin{aligned}
 &  |a_{i\dots i}-\alpha_i(\mathcal{A})|(|a_{j\dots j}-\alpha_j(\mathcal{A})|-\sum\limits_{\delta_{ij_2\dots j_m = 0}\atop\delta_{jj_2\dots j_m = 0}}|\alpha_j(\mathcal{A})-a_{jj_2\dots j_m}|)\\
      & >|\alpha_j(\mathcal{A})-a_{ji\dots i}|\sum\limits_{\delta_{ii_2\dots i_m = 0}}|\alpha_i(\mathcal{A})-a_{ii_2\dots i_m}|.
\end{aligned}
\end{equation}

if and only if
\begin{equation}\label{a2aa}
\begin{aligned}
 &  (b_{i\dots i}-\beta_i(\mathcal{B})(b_{j\dots j}-\beta_j(\mathcal{B})-\sum\limits_{\delta_{ij_2\dots j_m = 0}\atop\delta_{jj_2\dots j_m = 0}}(\beta_j(\mathcal{B})-b_{jj_2\dots j_m})\\
      & >(\beta_j(\mathcal{B})-b_{ji\dots i})\sum\limits_{\delta_{ii_2\dots i_m = 0}}(\beta_i(\mathcal{B})-b_{ii_2\dots i_m}).
\end{aligned}
\end{equation}
 This proves the result.
\end{proof}

\section{New bounds for $H$-eigenvalues of even order symmetric tensors}\label{eigen-region}

In this section, we discuss results about location of the $H$-eigenvalues of even order symmetric tensors.
First we state the following useful lemma. Proof is easy to verify.
\begin{lemma}\label{h6}
Let $\mathcal{A}=(a_{i_1\dots i_m})\in \mathcal{T}(\mathbb{R}^n,m)$. Then $0$ is an $H$-eigenvalue of $\mathcal{A}$ if and only if $0$ is an $H$-eigenvalue of $\overline{\mathcal{A}}$.
\end{lemma}

In the following theorem we give a new region for the $H$-eigenvalues of even order symmetric tensors.


\begin{theorem}\label{4ss}
Let $\mathcal{A}=(a_{i_1\dots i_m})\in \mathcal{T}(\mathbb{R}^n,m)$ be an even order symmetric tensor and $\lambda$ be an
$H$-eigenvalue
 of $\mathcal{A}$.  For any $i\in[n]$, Define
 $$\Lambda_i=[a_{i\dots i}-\beta_i(\mathcal{A}), a_{i\dots i}-\gamma_i(\mathcal{A})],$$
 $$\widetilde{\Lambda}_i=(a_{i\dots i}-\beta_i(\mathcal{A})-\Delta_i(\mathcal{A}), a_{i\dots i}-\gamma_i(\mathcal{A})-\Theta_i(\mathcal{A}))$$
 and for any $i\neq j\in[n]$,
 \begin{eqnarray*}
     &&\Lambda^1_{ij}=\{x\in(-\infty, \min\{a_{i\dots i}, a_{j\dots j}\} ): |a_{i\dots i}-\beta_i(\mathcal{A})-x||a_{j\dots j}-\beta_j(\mathcal{A})-x|\leq\Delta_i(\mathcal{A})\Delta_j(\mathcal{A})\},
 \end{eqnarray*}
 \begin{eqnarray*}
     \Lambda^2_{ij}=\{x\in(a_{i\dots i}, a_{j\dots j}): |a_{i\dots i}-\gamma_i(\mathcal{A})-x||a_{j\dots j}-\beta_j(\mathcal{A})-x|
      \leq\Theta_i(\mathcal{A})\Delta_j(\mathcal{A})\},
 \end{eqnarray*}
 \begin{eqnarray*}
     \Lambda^3_{ij}=\{x\in(a_{j\dots j}, a_{i\dots i}): |a_{i\dots i}-\beta_i(\mathcal{A})-x||a_{j\dots j}-\gamma_j(\mathcal{A})-x|
      \leq\Delta_i(\mathcal{A})\Theta_j(\mathcal{A})\},
 \end{eqnarray*}
 \begin{eqnarray*}
    && \Lambda^4_{ij}=\{x\in(\max\{a_{i\dots i},a_{j\dots j}\}, +\infty) : |a_{i\dots i}-\gamma_i(\mathcal{A})-x||a_{j\dots j}-\gamma_j(\mathcal{A})-x|\leq\Theta_i(\mathcal{A})\Theta_j(\mathcal{A})\},
 \end{eqnarray*}
where $\beta_i(\mathcal{A})$ and $\gamma_i(\mathcal{A})$ are defined in $(\ref{88})$ and $(\ref{99})$. Let
\begin{equation*}
\Lambda_{ij}=\left\{
   \begin{array}{ll}
     \Lambda^1_{ij}\cup\Lambda^2_{ij}\cup\Lambda^4_{ij}, & \text{if} \  \  a_{i\dots i}\leq a_{j\dots j}, \\
    \Lambda^1_{ij}\cup\Lambda^3_{ij}\cup\Lambda^4_{ij}, & \text{if} \  \  a_{i\dots i}\geq a_{j\dots j}.
   \end{array}
 \right.
\end{equation*}
Then,
\begin{equation}\label{sj}
\lambda\in\Lambda(\mathcal{A})=\big(\bigcup^n_{i=1}\Lambda_i\big)\cup\big(\bigcup^n_{i=1}\widetilde{\Lambda}_i\big)\cup\big(\bigcup^n_{i,j=1\atop i\neq j}\Lambda_{ij}\big).
\end{equation}
\end{theorem}

\begin{proof}
Let $\lambda$ be an $H$-eigenvalue of $\mathcal{A}$. Suppose that $\lambda\notin\Lambda(\mathcal{A})$.
Observe that $\mathcal{A}-\lambda\mathcal{I}$ and $\mathcal{A}$ have the same off-diagonal elements.  Since $\lambda \notin \Lambda_i$, we have either $a_{i\dots i} - \lambda > \beta_i(\mathcal{A})$ or $a_{i\dots i} - \lambda < \gamma_i(\mathcal{A})$ for all $i \in [n]$. Thus $|a_{i\dots i}-\lambda|>|\alpha_i(\mathcal{A})| = |\alpha_i(\mathcal{A} - \lambda \mathcal{I})|$ for all $i \in [n]$. Since $\lambda \notin \widetilde{\Lambda}_i$, we have either $(a_{i\dots i} - \beta_{i}(\mathcal{A}) - \Delta_i(\mathcal{A}) -\lambda) \geq 0 $ or $(a_{i\dots i} - \gamma_{i}(\mathcal{A}) - \Theta_i(\mathcal{A}) -\lambda) \leq 0$. Thus $|a_{i\dots i}-\lambda - \alpha_i(\mathcal{A}-\lambda \mathcal{I})| \geq \sum\limits_{\delta_{ii_2\dots i_m = 0}} |\alpha_i(\mathcal{A} -\lambda \mathcal{I})-a_{ii_2\dots i_m}|$.

Let $i, j \in [n]$ and $i \neq j$. We may assume without loss of generality that $a_{i\dots i} \leq a_{j\dots j}.$ If $\lambda \in (-\infty, a_{i\dots i})$, then $a_{i \dots i} - \lambda > 0$, $a_{j \dots j} - \lambda > 0$ and

$$ |a_{i\dots i}-\beta_i(\mathcal{A})-\lambda||a_{j\dots j}-\beta_j(\mathcal{A})-\lambda| > \Delta_i(\mathcal{A})\Delta_j(\mathcal{A}).$$

If $\lambda \in (a_{i\dots i},a_{j\dots j})$, then $a_{i \dots i} - \lambda < 0$, $a_{j \dots j} - \lambda > 0$ and

$$ |a_{i\dots i}-\gamma_i(\mathcal{A})-\lambda||a_{j\dots j}-\beta_j(\mathcal{A})-\lambda|
      > \Theta_i(\mathcal{A})\Delta_j(\mathcal{A}).$$

If $\lambda \in (a_{j\dots j}, \infty)$, then $a_{i \dots i} - \lambda < 0$, $a_{j \dots j} - \lambda < 0$ and
     $$ |a_{i\dots i}-\gamma_i(\mathcal{A})-\lambda||a_{j\dots j}-\gamma_j(\mathcal{A})-\lambda|> \Theta_i(\mathcal{A})\Theta_j(\mathcal{A}).$$

From the above observations, we have the following inequality
\begin{equation*}
\begin{aligned}
&|a_{i\dots i} -\lambda - \alpha_i(\mathcal{A}-\lambda \mathcal{I})||a_{j\dots j} -\lambda - \alpha_j(\mathcal{A}-\lambda \mathcal{I})| \\& > (\sum\limits_{\delta_{ii_2\dots i_m = 0}} |\alpha_i(\mathcal{A}-\lambda \mathcal{I})-a_{ii_2\dots i_m}|)(\sum\limits_{\delta_{ji_2\dots i_m = 0}} |\alpha_j(\mathcal{A}-\lambda \mathcal{I})-a_{ji_2\dots i_m}|).
\end{aligned}
\end{equation*}
Thus, by Theorem \ref{ta5}, we have $\mathcal{A}-\lambda\mathcal{I}$ is a double ${\overline{B}}$-tensor. By Theorem \ref{4l1} (a) and Lemma \ref{h6}, $0$ is not an $H$-eigenvalue of $\mathcal{A}-\lambda \mathcal{I}$. Thus there does not exist a nonzero vector $x \in \mathbb{R}^n$ such that $(\mathcal{A}-\lambda\mathcal{I})x^{m-1}=0$, which implies $\lambda$ is not an $H$-eigenvalue of $\mathcal{A}$. This is a contradiction. Therefore, $\lambda\in\Lambda(\mathcal{A})$.
\end{proof}

The above theorem is an analogous version of \cite[Theorem 3.3]{Pena1} for tensors. The following theorem is a refinement of the above theorem. The eigenvalue region we get in this theorem are smaller than the above one.

\begin{theorem}\label{4nn}
Let $\mathcal{A}=(a_{i_1\dots i_m})\in \mathcal{T}(\mathbb{R}^n,m)$ be an even order symmetric tensor and $\lambda$ be an
$H$-eigenvalue
 of $\mathcal{A}$.  Define
 $$\Psi_i=[a_{i\dots i}-\beta_i(\mathcal{A}), a_{i\dots i}-\gamma_i(\mathcal{A})], \  \  i\in[n]$$
 and for any $i\neq j\in[n]$,
 \begin{eqnarray*}
     &&\Psi^1_{ij}=\{x\in(-\infty, \min\{a_{i\dots i}, a_{j\dots j}\} ): |a_{i\dots i}-\beta_i(\mathcal{A})-x|(|a_{j\dots j}-\beta_j(\mathcal{A})-x|-\Delta_j^i(\mathcal{A}))\\
      &&\leq(\beta_j(\mathcal{A})-a_{ji\dots i})\Delta_i(\mathcal{A})\},
 \end{eqnarray*}
  \begin{eqnarray*}
     \Psi^2_{ij}=\{x\in(a_{i\dots i}, a_{j\dots j}): |a_{i\dots i}-\gamma_i(\mathcal{A})-x|(|a_{j\dots j}-\beta_j(\mathcal{A})-x|-\Delta_j^i(\mathcal{A}))
      \leq(\beta_j(\mathcal{A})-a_{ji\dots i})\Theta_i(\mathcal{A})\},
 \end{eqnarray*}
  \begin{eqnarray*}
     \Psi^3_{ij}=\{x\in(a_{j\dots j}, a_{i\dots i}): |a_{i\dots i}-\beta_i(\mathcal{A})-x|(|a_{j\dots j}-\gamma_j(\mathcal{A})-x|-\Theta_j^i(\mathcal{A}))
      \leq(a_{ji\dots i}-\gamma_j(\mathcal{A}))\Delta_i(\mathcal{A})\},
 \end{eqnarray*}
  \begin{eqnarray*}
    && \Psi^4_{ij}=\{x\in(\max\{a_{i\dots i},a_{j\dots j}\}, +\infty) : |a_{i\dots i}-\gamma_i(\mathcal{A})-x|(|a_{j\dots j}-\gamma_j(\mathcal{A})-x|-\Theta_j^i(\mathcal{A}))\\
      &&\leq(a_{ji\dots i}-\gamma_j(\mathcal{A}))\Theta_i(\mathcal{A})\},
 \end{eqnarray*}
where $\beta_i(\mathcal{A})$ and $\gamma_i(\mathcal{A})$ are defined in $(\ref{88})$ and $(\ref{99})$. Let
\begin{equation*}
\Psi_{ij}=\left\{
   \begin{array}{ll}
     \Psi^1_{ij}\cup\Psi^2_{ij}\cup\Psi^4_{ij}, & \text{if} \  \  a_{i\dots i}\leq a_{j\dots j}, \\
    \Psi^1_{ij}\cup\Psi^3_{ij}\cup\Psi^4_{ij}, & \text{if} \  \  a_{i\dots i}\geq a_{j\dots j}.
   \end{array}
 \right.
\end{equation*}
Then,
\begin{equation}\label{hj}
\lambda\in\Psi(\mathcal{A})=\big(\bigcup^n_{i=1}\Psi_i\big)\cup\big(\bigcup^n_{i,j=1\atop i\neq j}\Psi_{ij}\big).
\end{equation}
\end{theorem}

\begin{proof}
Let $\lambda$ be an $H$-eigenvalue of $\mathcal{A}$. Suppose that $\lambda\notin\Psi(\mathcal{A})$.
 Then for any $i\in[n]$, $|a_{i\dots i}-\lambda|>|\alpha_i(\mathcal{A}-\lambda \mathcal{I})|$.
  Let  $i\neq j\in[n]$. Without loss of generality, assume $a_{i\dots i}\leq a_{j\dots j}$. If $\lambda \in (-\infty, a_{i\dots i})$, then
\begin{equation*}
|a_{i\dots i}-\beta_i(\mathcal{A})-\lambda|(|a_{j\dots j}-\beta_j(\mathcal{A})-\lambda|-\Delta_j^i(\mathcal{A}))>(\beta_j(\mathcal{A})-a_{ji\dots i})\Delta_i(\mathcal{A}).
\end{equation*}
If $\lambda \in (a_{i\dots i},a_{j\dots j})$, then
\begin{eqnarray*}
 |a_{i\dots i}-\gamma_i(\mathcal{A})-\lambda|(|a_{j\dots j}-\beta_j(\mathcal{A})-\lambda|-\Delta_j^i(\mathcal{A}))>(\beta_j(\mathcal{A})-a_{ji\dots i})\Theta_i(\mathcal{A}).
 \end{eqnarray*}
 If $\lambda \in (a_{j\dots j}, \infty)$, then
\begin{eqnarray*}
     |a_{i\dots i}-\gamma_i(\mathcal{A})-\lambda|(|a_{j\dots
     j}-\gamma_j(\mathcal{A})-\lambda|-\Theta_j^i(\mathcal{A})>(a_{ji\dots i}-\gamma_j(\mathcal{A}))\Theta_i(\mathcal{A}).
 \end{eqnarray*}
 Thus in all three cases we have,
 \begin{equation*}
\begin{aligned}
 &  |a_{i\dots i}-\lambda -\alpha_i(\mathcal{A} -\lambda \mathcal{I})|(|a_{j\dots j}- \lambda -\alpha_j(\mathcal{A}-\lambda\mathcal{I})|-\sum\limits_{\delta_{ij_2\dots j_m = 0}\atop\delta_{jj_2\dots j_m = 0}}|\alpha_j(\mathcal{A}-\lambda\mathcal{I})-a_{jj_2\dots j_m}|)\\
      & >|\alpha_j(\mathcal{A}-\lambda\mathcal{I})-a_{ji\dots i}|\sum\limits_{\delta_{ii_2\dots i_m = 0}}|\alpha_i(\mathcal{A}-\lambda\mathcal{I})-a_{ii_2\dots i_m}|.
\end{aligned}
\end{equation*}
By Theorem \ref{mk}, we have $\mathcal{A}-\lambda\mathcal{I}$ is a quasi-double ${\overline{B}}$-tensor. Also, by Theorem \ref{4l1} (b) and Lemma \ref{h6}, $0$ is not an $H$-eigenvalue of $\mathcal{A}-\lambda \mathcal{I}$. Thus there does not exit a nonzero vector $x \in \mathbb{R}^n$ such that $(\mathcal{A}-\lambda\mathcal{I})x^{m-1}=0$, which means $\lambda$ is not an $H$-eigenvalue of $\mathcal{A}$. This is a contradiction. Therefore, $\lambda\in\Psi(\mathcal{A})$.
\end{proof}

We call $\Lambda(\mathcal{A})$ and $\Psi(\mathcal{A})$ the double ${\overline{B}}$-intervals and quasi-double ${\overline{B}}$-intervals of $\mathcal{A}$, respectively.  Observe that the double ${\overline{B}}$-intervals of $\mathcal{A}$ and quasi-double ${\overline{B}}$-intervals have a nature similar to the Brauer-type eigenvalues inclusion set $\Omega(\mathcal{A})$ stated in Theorem \ref{s0g}. Now, a question is nature, when range is better for determining  the location of the  $H$-eigenvalues of $\mathcal{A}$. We will discuss this interesting question in the following section.

\section{Comparisons between Brauer-type eigenvalues inclusion set, double ${\overline{B}}$-intervals and quasi-double ${\overline{B}}$-intervals}\label{compare}

We begin this section with the following observation. Li et al., \cite[Proposition 4]{LL} proved that if $\mathcal{A}$ is a double $B$-tensor, then $\mathcal{A}$ is a quasi-double $B$-tensor. Thus we have, if $\mathcal{A}$ is a double $\overline{B}$-tensor, then $\mathcal{A}$ is a quasi-double $\overline{B}$-tensor. In the next theorem, we prove the  quasi-double ${\overline{B}}$-intervals is contained in the double ${\overline{B}}$-intervals.
\begin{theorem}\label{6g}
Let $\mathcal{A}=(a_{i_1...i_m})\in \mathcal{T}(\mathbb{R}^n,m)$  be an even order symmetric tensor and  $\lambda$ be an $H$-eigenvalue of $\mathcal{A}$.  Then,
\begin{equation}\label{kb}
\Psi(\mathcal{A})\subseteq\Lambda(\mathcal{A}),
\end{equation}
where $\Lambda(\mathcal{A})$ and $\Psi(\mathcal{A})$ are defined in $(\ref{sj})$ and $(\ref{hj})$, respectively.
\end{theorem}

\begin{proof}

Let $\lambda \notin \Lambda(\mathcal{A})$. Then, by Theorem \ref{ta5}, $\mathcal{A} - \lambda \mathcal{I}$ is a double $\overline{B}$-tensor. By the above observation, $\mathcal{A} - \lambda \mathcal{I}$ is a quasi-double $\overline{B}$-tensor. Thus, by Theorem \ref{mk}, we have $\lambda \notin \Psi(\mathcal{A})$. Hence $ \Psi(\mathcal{A}) \subseteq \Lambda(\mathcal{A})$.
\end{proof}

The following example shows that $\Psi(\mathcal{A})$ and Brauer-type eigenvalues inclusion set are not comparable.
\begin{example}\label{e2}
\emph{ Consider the symmetric tensors $\mathcal{A}_1=(a_{i_1i_2i_3i_4})$ and $\mathcal{A}_2=(b_{i_1i_2i_3i_4})$ of order 4 dimension 2 defined as follows:}
\begin{eqnarray*}
  &&a_{1111}=18, \  \ a_{2222}=20,   \ \  a_{1222}=a_{2122}=a_{2212}=a_{2221}=3,\\
  &&a_{1122}=a_{2211}=a_{1221}=a_{2112}=a_{2121}=a_{1212}=2,\\
   &&a_{1112}=a_{2111}=a_{1211}=a_{1121}=2,
\end{eqnarray*}
\emph{and}
\begin{eqnarray*}
  &&b_{1111}=2,  \ \  b_{2222}=6,   \ \  b_{1222}=b_{2122}=b_{2212}=b_{2221}=4,\\
  &&b_{1122}=b_{2211}=b_{1221}=b_{2112}=b_{2121}=b_{1212}=-2,\\
   &&b_{1112}=b_{2111}=b_{1211}=b_{1121}=5.
\end{eqnarray*}

 \emph{By using the methods of  Theorem \ref{s0g} and Theorem \ref{4nn}, one can get the location of $H$-eigenvalues of $\mathcal{A}_1$ and $\mathcal{A}_2$ listed in the following Table 1.}
\begin{center}
\noindent{ \begin{tabular}{lllll} \multicolumn{5}{c}
{{\bf Table
1}: \ Comparisons of Brauer-type eigenvalues inclusion set and quasi-double ${\overline{B}}$-intervals} \\

\hline

 && $\mathcal{A}_1$ && $\mathcal{A}_2$\\
\hline

Brauer-type eigenvalues inclusion set   & ~~~ & $[3, \ 36.6119]$       &  ~~~   & $[-22.2560, \ 28.6844]$               \\

 quasi-double ${\overline{B}}$-intervals   & ~~~& $[9, \ 36.6119]$    & ~~~ &  $[-24.9257, \ 32.6068]$   \\

$H$-eigenvalues & ~~~& $15$, \ \ $35.1469$    & ~~~ & $-20.2289$, \  \  $16.0666$    \\
 \hline

\end{tabular}}
\end{center}

\emph{As we can see in the table above, the quasi-double ${\overline{B}}$-intervals is better when estimating the location of $H$-eigenvalues of $\mathcal{A}_1$ while the Brauer-type eigenvalues inclusion set is more precise for locating the  $H$-eigenvalues of  $\mathcal{A}_2$.}

\emph{So, there is no indication that $\Psi(\mathcal{A})$ is tighter than $\Omega(\mathcal{A})$ or $\Omega(\mathcal{A})$ is tighter than $\Psi(\mathcal{A})$. They supplement each other.} $\Box$
\end{example}


Now, we construct a more precise region by using the results of Theorem \ref{s0g} and Theorem~\ref{4nn}.
\begin{theorem}\label{ff3}
Let $\mathcal{A}=(a_{i_1\dots i_m})\in \mathcal{T}(\mathbb{R}^n,m)$ be an even order symmetric tensor and $\Omega(\mathcal{A})$, $\Psi(\mathcal{A})$ be defined in $(\ref{dnnd})$, $(\ref{hj})$, respectively. If $\lambda$ is an $H$-eigenvalue of $\mathcal{A}$, then
\begin{equation}\label{ff4}
\lambda\in\Upsilon(\mathcal{A})=\Omega(\mathcal{A})\cap\Psi(\mathcal{A}).
\end{equation}
\end{theorem}

\section{Concluding remarks}\label{conclude}
We introduced two new classes of tensors called  double $\overline{B}$-tensors and quasi-double $\overline{B}$-tensors and established some of their properties. Using the properties, we derived two  regions $\Lambda(\mathcal{A})$, $\Psi(\mathcal{A})$ called double $\overline{B}$-intervals, quasi-double ${\overline{B}}$-intervals, respectively. We observed that quasi-double ${\overline{B}}$-intervals are smaller than double ${\overline{B}}$-intervals. We discussed the relationship between Brauer-type eigenvalues inclusion set $\Omega(\mathcal{A})$ and the quasi-double ${\overline{B}}$-intervals  $\Psi(\mathcal{A})$ and constructed the region $\Upsilon(\mathcal{A})$, which is more precise for the location of the $H$-eigenvalues of even order symmetric tensors. We investigated the location of $H$-eigenvalues of even order symmetric tensors. Do the double ${\overline{B}}$-intervals $\Lambda(\mathcal{A})$ and the quasi-double ${\overline{B}}$-intervals $\Psi(\mathcal{A})$ hold for even order nonsymmetric tensors and odd order tensors?  In the future, we will research these problems.

\section*{Acknowledgement}

The authors would like to thank Prof. Yaotang Li for his useful comments. The second author is supported by the German-Israeli Foundation for Scientific Research and Development (GIF) grant no.\ 1135-18.6/2011.

\bibliographystyle{amsplain}
\bibliography{Brauer}
\end{document}